\documentclass[12pt]{amsart}
\usepackage[T1]{fontenc}
\usepackage{amsmath, amssymb}
\usepackage[english]{babel}
\usepackage{mathrsfs}
\usepackage{enumitem}
\usepackage{url}

\newtheorem{thm}{Theorem}[section]
\newtheorem*{thm*}{Theorem}
\newtheorem*{cor*}{Corollary}
\newtheorem{conj}[thm]{Conjecture}
\newtheorem{lem}[thm]{Lemma}
\newtheorem{prop}[thm]{Proposition}
\newtheorem{cor}[thm]{Corollary}
\theoremstyle{definition}
\newtheorem{defn}[thm]{Definition}

\theoremstyle{remark}
\newtheorem{rem}[thm]{Remark}
\newtheorem{fact}[thm]{Fact}

\theoremstyle{plain}
\newtheorem{maintheorem}{Theorem}

\newcommand{\C}{\mathbb C}
\newcommand{\D}{\mathcal D}

\newcommand{\N}{\mathbb N}
\newcommand{\R}{\mathbb R}

\newcommand{\tv}[1]{[\![#1]\!]}
\newcommand{\meet}{\mathbin{\wedge}}
\newcommand{\join}{\mathbin{\vee}}
\newcommand{\restr}{\mathbin{\upharpoonright}}

\DeclareMathOperator{\Proj}{Proj}

\DeclareMathOperator{\ran}{ran}

\DeclareMathOperator{\Ad}{Ad}
\newcommand{\cstar}{$\mathrm{C}^*$}
\newcommand{\wstar}{$\mathrm{W}^*$}
\newcommand{\AWstar}{$\mathrm{AW}^*$}
\DeclareMathOperator{\dom}{dom}

\usepackage{todonotes}

\usepackage[bookmarks=true, colorlinks=true, linkcolor=blue, citecolor=red, hypertexnames=false, hyperindex, breaklinks]{hyperref}

\title[Every \AWstar-Algebra is Normal]{Every \AWstar-Algebra is Normal}

\author{Jananan Arulseelan and James E. Hanson}
\address{Department of Mathematics, Iowa State University, 396 Carver Hall, 411 Morrill Road, Ames, IA 50011, USA}
\email{jananan@iastate.edu}
\urladdr{https://sites.google.com/view/jananan-arulseelan}

\address{Department of Mathematics, Iowa State University, 396 Carver Hall, 411 Morrill Road, Ames, IA 50011, USA}
\email{jameseh@iastate.edu}
\urladdr{https://james-hanson.github.io/}

\begin{document}

\begin{abstract}
Using set-theoretic methods, we prove that every \AWstar-algebra is normal, resolving a question of Wright that has stood open for 46 years.  This was previously known in the case of \AWstar-factors, by work of Sait\^o and Wright. We also show that if there is a model $V$ of $\mathrm{ZFC}$ with an \AWstar-algebra that fails to be monotone complete, then in some forcing extension of $V$ there is an \AWstar-factor that fails to be monotone complete, implying that any $\mathrm{ZFC}$ proof that all \AWstar-factors are monotone complete yields a $\mathrm{ZFC}$ proof that all \AWstar-algebras are monotone complete. Both results build on transfer principles for Boolean-valued factor representations originally developed by Ozawa. This work was assisted by the Danus LLM orchestration system. 
\end{abstract}

\maketitle

\section{Introduction}\label{sec:introduction}

\subsection{Statements of Results}

Kaplansky introduced \AWstar-algebras in 1951 as an algebraic analogue of \wstar-algebras: the idea is to capture the good order-theoretic behaviour of von Neumann algebras purely algebraically, without reference to a representing Hilbert space or to topology. 

We use the Baer $*$-ring formulation of an \AWstar-algebra.

\begin{defn}
    A \cstar-algebra $A$ is an \AWstar-algebra if, for every subset $X\subseteq A$, there is a projection $e\in A$ such that
    \[
        R_A(X):=\{a\in A: xa=0 \text{ for every } x\in X\}=eA.
    \]
\end{defn}
In line with the motivation stated above, every von Neumann algebra satisfies the annihilator condition (a fact one can prove using the Double Commutant Theorem).  Note that \AWstar-algebras are always unital, and that the projection lattice $\Proj(A)$ of an \AWstar-algebra is always complete.

Kaplansky's original definition of \AWstar-algebras \cite{Kap51} differs from ours slightly, but is equivalent. He defined a \cstar-algebra $A$ to be an \AWstar-algebra if every set of orthogonal projections in $A$ admits a supremum in the partial order on projections and every MASA in $A$ is generated by its projections.  It can be shown that both conditions can be replaced by the condition that every MASA is monotone complete \cite[Proposition 1.4, Theorem 1.5]{SW15}.  An upward-directed family is understood to be nonempty.  

\begin{defn}\label{defn:monotone-complete}
    A \cstar-algebra $A$ is \textbf{monotone complete} if every norm-bounded upward-directed family in $A_{\mathrm{sa}}$ has a least upper bound in $A_{\mathrm{sa}}$.
\end{defn}

We phrase it in terms of norm-bounded families because that is the form in which it is used in Section~\ref{sec:monotone}.

An important open problem that has been studied since near the beginning of the theory of \AWstar-algebras is the monotone completeness conjecture:

\begin{conj}\label{conj:monotone}
    Every \AWstar-algebra is monotone complete.
\end{conj}

The conjecture has a clear affirmative answer for \wstar-algebras and for commutative \AWstar-algebras.  Various partial affirmative answers have been obtained \cite{AM08}, \cite{CP84}.

We use the following notion of normality, due to Wright (see \cite{W80}, \cite{SW91}, and \cite[Definition~5.12]{Gow26}).

\begin{defn}\label{defn:normal}
    An \AWstar-algebra $A$ is \textbf{normal} if, for every upward-directed family $(p_i)_{i\in I}$ in $\Proj(A)$, its join
    \[
        p=\bigvee\nolimits_{\Proj(A)}(p_i)_{i\in I}
    \]
    is the least upper bound of the family in the self-adjoint order $A_{\mathrm{sa}}$.
    Equivalently, if $a\in A_{\mathrm{sa}}$ and $a\geq p_i$ for every $i\in I$, then $a\geq p$.
\end{defn}

In other words, normality asks that the projection-lattice join and the self-adjoint-order supremum agree on upward-directed families of projections.  Given how closely these two orders are related, one might expect this to be free.  It is not.

Normality was introduced by Wright \cite{W80} as a technically weaker substitute for monotone completeness.  In the same paper, he shows that every \AWstar-algebra of finite type is normal; see also \cite{Sai81}.  Sait\^o and Wright later showed, in 1991, that every \AWstar-factor is normal (see Theorem~\ref{thm:factor-normality} below).  With this notion in hand, we can now state the first of our two main results.

\begin{maintheorem}\label{thm:normality}
    Every \AWstar-algebra is normal.
\end{maintheorem}

This resolves a question posed by Wright \cite{W80} in the affirmative.  Interestingly, immediately after posing the question, Wright posits, due to the seemingly poor behaviour of \AWstar-algebras, that the answer is probably negative, a guess that this paper does not bear out.

Previously, Ara-Goldstein \cite{AraGoldstein95} showed, via different methods, that Rickart \cstar-algebras (a class containing all \AWstar-algebras) are all $\sigma$-normal (a weakening of normality where the directed family of projections is assumed countable). Ara-Mathieu \cite[Theorem 5.6]{AM08} showed that every $\sigma$-finite \AWstar-algebra is normal.

Berberian \cite[Theorem 3]{Ber83} proved that an \AWstar-algebra is a \wstar-algebra if and only if it is normal and has a large \wstar-corner.  Here ``large'' refers to a corner with respect to a faithful projection (one with central support equal to the identity).  Since Theorem~\ref{thm:normality} removes the first of these two hypotheses, we have an immediate corollary.

\begin{cor}\label{cor:large-corner}
    An \AWstar-algebra $A$ is a \wstar-algebra if and only if it has a large \wstar-corner, that is, if and only if there is a faithful projection $p\in A$ such that $pAp$ is a \wstar-algebra.
\end{cor}

\begin{proof}
    If $A$ is a \wstar-algebra, take $p=1$.  Conversely, $A$ is normal by Theorem~\ref{thm:normality}, so \cite[Theorem 3]{Ber83} applies.
\end{proof}

Our second main result, whose proof follows a similar path to the first, reduces finding a $\mathrm{ZFC}$-proof of the monotone completeness conjecture to finding a $\mathrm{ZFC}$-proof of the \AWstar-factor case at the level of consistency strength.

\begin{maintheorem}\label{thm:monotone-reduction}
    There exists a model of $\mathrm{ZFC}$ containing an \AWstar-algebra that is not monotone complete if and only if there exists a model of $\mathrm{ZFC}$ containing an \AWstar-factor that is not monotone complete.
\end{maintheorem}

One direction of this equiconsistency is trivial, since an \AWstar-factor is in particular an \AWstar-algebra.  The content is the other direction: if $\mathrm{ZFC}+$ ``there is an \AWstar-algebra that is not monotone complete'' is consistent, then so is $\mathrm{ZFC}+$ ``there is an \AWstar-factor that is not monotone complete.''  By the contrapositive, if $\mathrm{ZFC}$ proves that every \AWstar-factor is monotone complete, then $\mathrm{ZFC}$ proves Conjecture~\ref{conj:monotone}.  So a proof of the monotone completeness conjecture for factors, formalizable in $\mathrm{ZFC}$, is already a proof of the conjecture in general; and any consistency or independence proof for a counterexample is already one for a factor counterexample.

\subsection{Proof Outline}

Both theorems follow the same basic strategy, borrowed from Ozawa's Boolean-valued analysis of \AWstar-algebras.  The idea is to absorb the full center of $A$ into the Boolean-valued scalar field, so that the bounded-section representation (\cite[Theorem~5]{Oza86}) turns $A$ into the bounded global-section algebra of a single internal \AWstar-factor $M$.  Since factors are already well understood (the factor normality theorem \cite[Corollary~4.7]{SW91} settles normality outright, and failure of monotone completeness of factors is exactly the statement that our second theorem shows to be equiconsistent with the failure of the general conjecture), the real content of the paper is in showing that the specific properties we care about cross the Boolean-valued boundary cleanly, with nothing gained or lost in either direction.  Section~\ref{sec:transfer} carries this out for normality: central mixing, directed ascent, identification of the internal join, ascent of arbitrary self-adjoint upper bounds, and exact descent of order.  Section~\ref{sec:monotone} carries out the analogous crossing for monotone completeness, one level up in generality, at arbitrary bounded self-adjoint directed families rather than just projections.

We remark that our strategy of applying transfer principles and Boolean-valued set theory is far from new.  Indeed, Ozawa and Sait\^o \cite{OzawaSaito86} applied the same technique to internal von Neumann algebras to derive results about classes of so-called embeddable \AWstar-algebras (see also \cite{Oza85}).  Takeuti \cite{Tak83} used similar techniques to transfer statements from von Neumann factors to von Neumann algebras with center.  For a recent use of forcing over a measure algebra in a nearby operator-algebraic setting, where the global sections are the randomization $L^1(\Omega,M)$, see Farah and Vaccaro \cite{FV25}.  

\begin{rem}
    The scrupulous reader may ask why the conjecture remained open this long when the main inputs to our proof are over 30 years old.  Indeed, the onus is on the authors to explain why our results are not merely corollaries of Theorem \ref{thm:factor-normality} and the Ozawa machinery.  While we can only speculate, we point to a particular subtlety that obstructs a straightforward composition of results and that our proof addresses.  Ozawa states the template we follow in \cite[\S4]{Oza86}: if $\phi$ is a theorem on \AWstar-factors then $\tv{\phi}=1$ is again a theorem of $\mathrm{ZFC}$, and by his Theorem~5 this is equivalent to some statement about \AWstar-algebras with center $Z$.  The precise statement you get is not transparent.  For instance, it is not clear that external directedness passes to internal directedness (similarly for external joins and upper bounds) because this is about the semantics of the internal sets.  These internal sets are the result of mixing with Boolean algebra elements.  Thus, it is only after proving Lemma \ref{lem:directed-ascent}, Lemma \ref{lem:annihilator} and Lemma \ref{lem:descent-normality} that we can apply transfer to conclude Theorem \ref{thm:normality}.  These steps are each non-trivial to see.
    
\end{rem}

\subsection*{Acknowledgements}  The first author thanks Matthew Kennedy for pointing them to the possibility of a set-theoretic resolution to the monotone completeness conjecture, and for stimulating early conversations on the topic.  

\subsection*{AI disclosure}

The ideas for the proofs in this paper were found by using the Danus LLM orchestration system \cite{LiuEtAl26} (using Fable 5.1 and GPT 5.6 Sol) to explore set-theoretic approaches to problems around monotone completeness. This paper was written by the authors, who take full responsibility for its correctness.

\section{Background}\label{sec:background}

\subsection{Notation and Conventions}

We fix some notation used for the rest of the paper.  Throughout, $A$ denotes an \AWstar-algebra, $Z=Z(A)$ denotes its center, and $B=\Proj(Z)$ is the lattice of projections of the center.  The Scott--Solovay universe over $B$ is denoted $V^{B}$, Boolean truth values are written $\tv{\varphi}$, and once a bounded-section representation is fixed, $M$ denotes the internal algebra.  These conventions are in force in this section and in \S\ref{sec:introduction}.  The results of \S\ref{sec:transfer} and \S\ref{sec:monotone} are stated for an arbitrary complete Boolean algebra and an arbitrary internal algebra, and re-declare their hypotheses as they go.

\subsection{Boolean-valued models of set theory}\label{subsec:bvm}

Boolean-valued semantics generalizes the two-element algebra of truth values to an arbitrary complete Boolean algebra $B$.  From the ground model one builds a universe $V^{B}$ of $B$-valued sets, and each sentence of the language of set theory, with parameters drawn from $V^{B}$, is assigned a truth value $\tv{\varphi}\in B$; the sentences receiving value $1$ are the ones that hold in $V^{B}$.  When $B$ is the regular-open algebra of a forcing poset this is forcing, and $\tv{\varphi}$ is the largest condition forcing $\varphi$.  We shall not need a generic filter, and work with the algebra of truth values alone.

The universe $V^{B}$ is defined by recursion on rank.  Let $V^{B}_{\alpha}$ consist of the functions $u$ with $\dom(u)\subseteq\bigcup_{\beta<\alpha}V^{B}_{\beta}$ and $\ran(u)\subseteq B$, and let $V^{B}=\bigcup_{\alpha}V^{B}_{\alpha}$.  A \textbf{name} $u$ is to be thought of as a set together with, for each candidate member $w$, a Boolean value $u(w)$ recording the extent to which $w$ belongs to it.  Writing $a\Rightarrow c$ for $(1-a)\join c$, the atomic truth values are given by the simultaneous recursion
\[
    \tv{u\in v}=\bigvee_{w\in\dom(v)}\bigl(v(w)\meet\tv{u=w}\bigr),
\]
\[
    \tv{u=v}=\bigwedge_{w\in\dom(u)}\bigl(u(w)\Rightarrow\tv{w\in v}\bigr)
        \meet\bigwedge_{w\in\dom(v)}\bigl(v(w)\Rightarrow\tv{w\in u}\bigr),
\]
and extended to all formulas by $\tv{\neg\varphi}=1-\tv{\varphi}$, $\tv{\varphi\wedge\psi}=\tv{\varphi}\meet\tv{\psi}$, $\tv{\forall x\,\varphi(x)}=\bigwedge_{u\in V^{B}}\tv{\varphi(u)}$ and $\tv{\exists x\,\varphi(x)}=\bigvee_{u\in V^{B}}\tv{\varphi(u)}$.  The last two range over a proper class, but each is attained on a set of names, so the definition is legitimate; see \cite[\S1]{Bell05} or \cite[Ch.~14]{Jech03}.

Each set $x$ of the ground model has a \textbf{canonical name} $\check{x}$, defined by $\dom(\check{x})=\{\check{y}:y\in x\}$ with all values $1$.  We reserve the check for this operation and the hat for the passage, described in \S\ref{subsec:sections}, from an element of a global-section algebra to the name representing it.

We collect some facts we use.  They are all standard.

\begin{fact}\label{fact:bvm}
Let $B$ be a complete Boolean algebra.
\begin{enumerate}
    \item[(i)] \textup{(Equality and substitutivity)} $\tv{u=u}=1$ and $\tv{u=v}\meet\tv{v=w}\leq\tv{u=w}$, and for every formula $\varphi$,
    \[
        \tv{u=v}\meet\tv{\varphi(u)}\leq\tv{\varphi(v)} .
    \]
    \item[(ii)] \textup{(Scott--Solovay)} Every axiom of $\mathrm{ZFC}$ has Boolean truth value $1$ in $V^{B}$.
    \item[(iii)] \textup{(Transfer principle)} Boolean-valued semantics is sound for first-order logic, so by \textup{(ii)}, if $\mathrm{ZFC}\vdash\varphi$ then $\tv{\varphi}=1$.
    \item[(iv)] \textup{(Maximum principle)} For every formula $\varphi$ there is a name $u$ with $\tv{\exists x\,\varphi(x)}=\tv{\varphi(u)}$.
    \item[(v)] \textup{(Mixing principle)} For every partition $(b_i)_{i\in I}$ of $1$ in $B$ and every family of names $(u_i)_{i\in I}$ there is a name $u$ with $b_i\leq\tv{u=u_i}$ for all $i$; any two such names $u,u'$ satisfy $\tv{u=u'}=1$.
\end{enumerate}
\end{fact}

\begin{proof}
For (i) see \cite[\S1]{Bell05}.  For (ii) and (iii) see \cite[Theorems~13.12 and~14.25]{TakeutiZaring73}, and for (iv) and (v) see \cite[Theorem~16.2 and Corollary~16.4]{TakeutiZaring73}.
\end{proof}

Part (iii) applies to theorems of $\mathrm{ZFC}$, and not to every sentence that happens to hold in the ground model.  If $V\models\mathsf{CH}$ and $B$ is the completion of the poset adding $\aleph_2$ Cohen reals, then $\mathsf{CH}$ holds in the ground model while $\tv{\mathsf{CH}}=0$.  We apply (iii) in \S\ref{sec:main-proof} to Theorem~\ref{thm:factor-normality}, which is a theorem of $\mathrm{ZFC}$, and once more in Remark~\ref{rem:direct-transfer} under an explicit provability hypothesis.

Two further pieces of Boolean algebra recur below.  The first replaces a covering by a partition, which is what lets us analyse an arbitrary name lying in an ascended set piece by piece.

\begin{lem}[Disjoint refinement]\label{lem:refinement}
Let $B$ be a complete Boolean algebra, let $(c_i)_{i\in I}$ be a family in $B$, and let $b\leq\bigvee_{i}c_i$.  Then there is a pairwise disjoint family $(b_i)_{i\in I}$ with $b_i\leq b\meet c_i$ for every $i$ and $\bigvee_i b_i=b$.  Given two such families over the same $b$, their pairwise meets form a common refinement.
\end{lem}

\begin{proof}
Well-order $I$ and put $b_i=b\meet c_i\meet\bigl(1-\bigvee_{j<i}c_j\bigr)$.  If $j<i$ then $b_i\leq 1-c_j$ while $b_j\leq c_j$, so the family is pairwise disjoint.  Write $d=\bigvee_i b_i$ and suppose $e:=b\meet(1-d)$ were nonzero.  Since $e\leq b\leq\bigvee_i c_i$ there is a least $i$ with $e\meet c_i\neq 0$.  Minimality gives $e\meet c_j=0$, and hence $e\leq 1-c_j$, for every $j<i$, so $e\leq 1-\bigvee_{j<i}c_j$ by completeness.  Then $0\neq e\meet c_i\leq b\meet c_i\meet\bigl(1-\bigvee_{j<i}c_j\bigr)=b_i\leq d$, which contradicts $e\leq 1-d$.  The last assertion is immediate.
\end{proof}

The second lets us localize an argument to a band of $B$.  For $b\in B$ the principal ideal $B\!\restr\!b=\{c\in B:c\leq b\}$ is again a complete Boolean algebra, with $b$ as its greatest element.  Every $B$-name $u$ restricts to a $B\!\restr\!b$-name $u\!\restr\!b$, by the recursion $\dom(u\!\restr\!b)=\{w\!\restr\!b:w\in\dom(u)\}$ and $(u\!\restr\!b)(w\!\restr\!b)=b\meet u(w)$.

\begin{fact}[Restriction]\label{fact:restriction}
Let $b\in B$ be nonzero.  For every formula $\varphi$ and all names $u_1,\dots,u_n$,
\[
    \tv{\varphi(u_1\!\restr\!b,\dots,u_n\!\restr\!b)}^{V^{B\restr b}}
        = b\meet\tv{\varphi(u_1,\dots,u_n)}^{V^{B}} .
\]
In particular the restricted statement has truth value $b$, the unit of $B\!\restr\!b$, exactly when $b\leq\tv{\varphi}$ in $V^{B}$, and truth value $0$ exactly when $b\meet\tv{\varphi}=0$.
\end{fact}

Passing to $B\!\restr\!b$ is the Boolean-valued form of restricting attention to those generic filters that contain $b$, which is why truth values are cut down by $b$.  We suppress restriction from the notation and write $u$ for $u\!\restr\!b$.  Fact~\ref{fact:restriction} is used twice in \S\ref{sec:monotone}, once to pin down the largest band on which monotone completeness holds and once, in the proof of Theorem~\ref{thm:monotone-reduction}, to move to a band on which it fails outright.

We record one instance of the transfer principle explicitly.  The Archimedean property of the reals is a theorem of $\mathrm{ZFC}$, so if $R$ is a name with $\tv{R\in\R}=1$ then
\begin{equation}\label{eq:archimedean}
    \bigvee_{n\in\N}\tv{R\leq\check{n}}=1 .
\end{equation}
This is what allows an internally bounded family to be sliced into countably many bands, on each of which a ground-model bound is available.

A reader accustomed to forcing may expect the arguments below to depend on a chain condition on $B$.  They do not.  The worry is a reasonable one: we shall pass repeatedly between an upward-directed family in an algebra and its internal counterpart, and one might expect the latter to be directed differently, or to have a different cofinality, from the family it came from.  Everything we use about $V^{B}$ is contained in Fact~\ref{fact:bvm}, Lemma~\ref{lem:refinement} and Fact~\ref{fact:restriction}, all of which hold for an arbitrary complete Boolean algebra, and every construction below chooses its witnesses cell by cell over a partition of $B$ and mixes them.  No cardinality-preservation or cofinality-preservation hypothesis appears anywhere in this paper.

\subsection{Bounded global sections}\label{subsec:sections}

Boolean-valued analysis is useful here because an internal normed structure can be pushed back down to an ordinary one.  Let $N$ be a name with $\tv{N\text{ is a \cstar-algebra}}=1$, all norms being computed internally.  A name $u$ is a \textbf{bounded section} of $N$ if $\tv{u\in N}=1$ and $\tv{\|u\|\leq\check{K}}=1$ for some ground-model real $K$.  Two bounded sections are identified when $\tv{u=v}=1$, and the set of resulting classes is the \textbf{bounded global-section algebra} $A$ of $N$.  For $x\in A$ we write $\widehat{x}$ for a name representing it and call $\widehat{x}$ the \textbf{ascent} of $x$; conversely $x$ is the \textbf{descent} of $\widehat{x}$.  The operations of $A$ are defined by descent.  For instance $x+y$ is the unique element of $A$ with $\tv{\widehat{x+y}=\widehat{x}+\widehat{y}}=1$, which exists by Fact~\ref{fact:bvm}(iv) and is unique by the identification just made.  The norm is
\[
    \|x\|=\inf\{\alpha\in\R:\tv{\|\widehat{x}\|\leq\check{\alpha}}=1\} .
\]
This is Ozawa's construction \cite[\S2]{Oza86}, following Takeuti \cite{Tak83b}, and $A$ is again a \cstar-algebra.

Two points deserve emphasis before we go on.  First, boundedness is by a ground-model real.  An internal element of $N$ whose norm is bounded only by some internal real is not a bounded section, and this is the one place where the passage between $N$ and $A$ is not transparent.  It is the reason for the band decomposition of Theorem~\ref{thm:directed-band-form}.  Second, the scalars of $N$ are the complex numbers as computed in $V^{B}$, which properly contain $\check{\C}$.  The bounded global sections of the internal complex field therefore form a commutative \AWstar-algebra rather than a copy of $\C$.  We denote it by $Z_{B}$.  By \cite[Theorem~3.6 and Corollary~3.7]{Oza84} it is $*$-isomorphic to any commutative \AWstar-algebra whose projection lattice is isomorphic to $B$, and Lemma~\ref{lem:commutative} below reproves this in the form we use.

Ascent and descent are exact for everything we shall need.

\begin{fact}\label{fact:exactness}
Let $N$ be an internal \cstar-algebra with bounded global-section algebra $A$, and let $x,y\in A$.
\begin{enumerate}
    \item[(i)] $x=y$ if and only if $\tv{\widehat{x}=\widehat{y}}=1$, and the algebraic operations and the involution commute with ascent with Boolean truth value $1$.
    \item[(ii)] For $\alpha\in\R$, $\|x\|\leq\alpha$ if and only if $\tv{\|\widehat{x}\|\leq\check{\alpha}}=1$.
    \item[(iii)] The element $x$ is self-adjoint, positive, or a projection if and only if we have $\tv{\widehat{x}\text{ is self-adjoint}}=1$, $\tv{\widehat{x}\text{ is positive}}=1$, or $\tv{\widehat{x}\text{ is a projection}}=1$ respectively.  Consequently $x\leq y$ in $A_{\mathrm{sa}}$ if and only if $\tv{\widehat{x}\leq\widehat{y}}=1$.
    \item[(iv)] Every name $u$ with $\tv{u\in N}=1$ and $\tv{\|u\|\leq\check{K}}=1$ for some ground-model real $K$ is the ascent of a unique element of $A$.
\end{enumerate}
\end{fact}

\begin{proof}
Parts (i), (ii) and (iv) are immediate from the construction.  Part (iii) is \cite[Proposition~3]{Oza86}, and the consequence follows by applying it to $y-x$.
\end{proof}

The next fact is the dictionary between the Boolean algebra and the central projections of $A$, and it is what gives phrases such as ``mix $u$ on $b$ and zero off $b$'' a meaning in $A$.  Once $Z_{B}$ is identified with a subalgebra of the center of $A$, each $b\in B$ is a central projection of $A$.

\begin{fact}\label{fact:dictionary}
Let $A$ be the bounded global-section algebra of $N$, regarded as containing $Z_{B}$ centrally, and let $x,y\in A$ and $b\in B$.  Then
\[
    b\leq\tv{\widehat{x}=\widehat{y}}\quad\text{if and only if}\quad bx=by .
\]
Consequently, if $(b_i)_{i\in I}$ is a partition of $1$ in $B$ and $(x_i)_{i\in I}$ is a norm-bounded family in $A$, then the name supplied by Fact~\ref{fact:bvm}(v) is the ascent of the unique $x\in A$ satisfying $b_ix=b_ix_i$ for all $i$, and $\|x\|\leq\sup_i\|x_i\|$.
\end{fact}

\begin{proof}
The displayed equivalence is the computation of $\tv{\Phi(x)=\Phi(y)}$ in the proof of \cite[Theorem~2]{Oza86}.  For the consequence, the mixed name has norm at most $\sup_i\|x_i\|$ and so is a bounded section, and Fact~\ref{fact:exactness}(iv) descends it.
\end{proof}

Restricting the algebra of truth values to a band restricts the algebra of bounded sections to the corresponding central corner.  We isolate this, since it is what allows a statement about $A$ to be localized to $bA$ by localizing $B$ to $B\!\restr\!b$, and since it is not contained in Fact~\ref{fact:restriction}, which speaks only of truth values.

\begin{lem}[Sections over a band]\label{lem:band-sections}
Let $N$ be a name with $\tv{N\text{ is a \cstar-algebra}}=1$, let $A$ be its bounded global-section algebra, and let $b\in B$ be nonzero.  Then $N\!\restr\!b$ is internally a \cstar-algebra in $V^{B\restr b}$, and the assignment
\[
    \Theta(x)=\text{the class of }\widehat{x}\!\restr\!b
\]
is an isometric $*$-isomorphism from $bA$ onto the bounded global-section algebra of $N\!\restr\!b$, carrying the unit $b$ of $bA$ to the internal unit.
\end{lem}

\begin{proof}
That $N\!\restr\!b$ is internally a \cstar-algebra is Fact~\ref{fact:restriction} applied to the formula defining $N$.  We use throughout that checks commute with restriction: the canonical name of a ground-model set $y$ in $V^{B\restr b}$ has all values $b$, and so does $\check{y}\!\restr\!b$, so induction on rank identifies the two.  Fact~\ref{fact:restriction} may therefore be applied to formulas carrying ground-model parameters without further comment.

Let $x\in A$ and choose a ground-model real $K\geq0$ with $\tv{\|\widehat{x}\|\leq\check{K}}=1$.  By Fact~\ref{fact:restriction},
\[
    \tv{\widehat{x}\!\restr\!b\in N\!\restr\!b}^{V^{B\restr b}}=b,
    \qquad
    \tv{\|\widehat{x}\!\restr\!b\|\leq\check{K}}^{V^{B\restr b}}=b,
\]
so $\widehat{x}\!\restr\!b$ is a bounded section of $N\!\restr\!b$ and $\Theta(x)$ is defined.

Suppose $x,y\in bA$.  By Fact~\ref{fact:restriction}, $\tv{\widehat{x}\!\restr\!b=\widehat{y}\!\restr\!b}^{V^{B\restr b}}=b$ if and only if $b\leq\tv{\widehat{x}=\widehat{y}}$, which by Fact~\ref{fact:dictionary} holds if and only if $bx=by$, that is, if and only if $x=y$.  So $\Theta$ is well defined and injective on $bA$.  For the norm, note that $x=bx$ gives $(1-b)x=(1-b)0$, so $1-b\leq\tv{\widehat{x}=\widehat{0}}$ by Fact~\ref{fact:dictionary}, whence $1-b\leq\tv{\|\widehat{x}\|\leq\check{\alpha}}$ for every $\alpha\geq0$.  Consequently $\tv{\|\widehat{x}\|\leq\check{\alpha}}=1$ if and only if $b\leq\tv{\|\widehat{x}\|\leq\check{\alpha}}$, and the two infima defining the norm of $x$ in $A$ and of $\Theta(x)$ in the bounded global-section algebra of $N\!\restr\!b$ agree.

For surjectivity we first record that a $B\!\restr\!b$-name is a $B$-name fixed by restriction.  Indeed, the values of such a name $u$ lie in $B\!\restr\!b$, so $b\meet u(w)=u(w)$ for every $w\in\dom(u)$, and induction on rank gives $u\!\restr\!b=u$.  Now let $u$ be a bounded section of $N\!\restr\!b$, say with $\tv{\|u\|\leq\check{K}}^{V^{B\restr b}}=b$ for a ground-model $K\geq0$.  Reading $u$ as a $B$-name and applying Fact~\ref{fact:restriction} to $u=u\!\restr\!b$ gives
\[
    b\leq\tv{u\in N},\qquad b\leq\tv{\|u\|\leq\check{K}} .
\]
Let $v$ be a name mixing $u$ on $b$ with $\widehat{0}$ on $1-b$, as in Fact~\ref{fact:bvm}(v).  Then $\tv{v\in N}=1$ and $\tv{\|v\|\leq\check{K}}=1$, so by Fact~\ref{fact:exactness}(iv) there is $x\in A$ with $\tv{\widehat{x}=v}=1$.  From $1-b\leq\tv{\widehat{x}=\widehat{0}}$ and Fact~\ref{fact:dictionary} we get $(1-b)x=0$, so $x\in bA$; and from $b\leq\tv{\widehat{x}=u}$ and Fact~\ref{fact:restriction} we get $\tv{\widehat{x}\!\restr\!b=u}^{V^{B\restr b}}=b$.  Hence $\Theta(x)$ is the class of $u$.

Finally, the operations on both sides are defined by descent, so $\Theta$ matches them.  For instance $\tv{\widehat{x+y}=\widehat{x}+\widehat{y}}=1$ restricts to $\tv{\widehat{x+y}\!\restr\!b=\widehat{x}\!\restr\!b+\widehat{y}\!\restr\!b}^{V^{B\restr b}}=b$, whence $\Theta(x+y)=\Theta(x)+\Theta(y)$; multiplication, the involution and the action of $Z_{B}\!\restr\!b$ are identical.  For the unit, $b\cdot b=b\cdot1$ gives $b\leq\tv{\widehat{b}=1_N}$ by Fact~\ref{fact:dictionary}.
\end{proof}

\begin{rem}\label{rem:band-transfer}
By Fact~\ref{fact:restriction}, if $M$ is an internal \AWstar-factor in $V^{B}$ then $M\!\restr\!b$ is an internal \AWstar-factor in $V^{B\restr b}$, and Lemma~\ref{lem:band-sections} identifies its bounded global-section algebra with $bA$.  Since $\Theta$ is a $*$-isomorphism, the discussion of Remark~\ref{rem:iso-invariance} applies to it verbatim, so normality, monotone completeness, projection-lattice joins and self-adjoint suprema may be read off on either side.  In keeping with the convention of \S\ref{subsec:bvm} we suppress the restriction and write $M$ for $M\!\restr\!b$.
\end{rem}

Internal quantification over $N$ can be computed over the bounded sections alone, which keeps the unbounded part of $N$ from interfering.

\begin{fact}[{\cite[Lemma~1]{Oza86}}]\label{fact:quantifiers}
Let $A$ be the bounded global-section algebra of $N$ and let $\varphi,\psi$ be formulas with one free variable.  If $\tv{\varphi(\widehat{u})}=1$ for some $u\in A$, then
\[
    \tv{\forall x\in N\,(\varphi(x)\Rightarrow\psi(x))}
        =\bigwedge\{\tv{\psi(\widehat{x})}:x\in A,\ \tv{\varphi(\widehat{x})}=1\},
\]
\[
    \tv{\exists x\in N\,(\varphi(x)\wedge\psi(x))}
        =\bigvee\{\tv{\psi(\widehat{x})}:x\in A,\ \tv{\varphi(\widehat{x})}=1\} .
\]
\end{fact}

It remains to say how an external subset of $A$ becomes an internal subset of $N$.  For $X\subseteq A$, let $\dot X$ be the name with $\dom(\dot X)=\{\widehat{x}:x\in X\}$ and all Boolean values $1$.  Unwinding the definition of $\tv{\in}$ gives, for every name $y$,
\begin{equation}\label{eq:membership}
    \tv{y\in\dot X}=\bigvee_{x\in X}\tv{y=\widehat{x}} .
\end{equation}

The following description of the members of $\dot X$ is worth keeping in mind throughout \S\ref{sec:transfer} and \S\ref{sec:monotone}.  The set $\dot X$ is the closure of $X$ under central mixing, so a property asserted of it internally has to be checked against the mixtures and not only against the displayed ascents.

\begin{lem}\label{lem:mixtures}
Let $X\subseteq A$ be nonempty, let $y$ be a name and let $b\in B$.  Then $b\leq\tv{y\in\dot X}$ if and only if there are a partition $(b_i)_{i\in I}$ of $b$ and elements $x_i\in X$ with $b_i\leq\tv{y=\widehat{x_i}}$ for every $i$.  In particular $\tv{y\in\dot X}=1$ if and only if such a family exists with $(b_i)$ a partition of $1$.
\end{lem}

\begin{proof}
If such a family exists then $b=\bigvee_ib_i\leq\bigvee_{x\in X}\tv{y=\widehat{x}}=\tv{y\in\dot X}$ by \eqref{eq:membership}.  Conversely, if $b\leq\tv{y\in\dot X}$, apply Lemma~\ref{lem:refinement} to the family $\bigl(\tv{y=\widehat{x}}\bigr)_{x\in X}$.
\end{proof}

By Fact~\ref{fact:dictionary}, when $X$ is norm bounded the descents of the members of $\dot X$ are exactly the central mixtures of members of $X$ inside $A$.  For $X\subseteq\Proj(A)$ no boundedness hypothesis is needed, since projections have norm at most $1$.

\subsection{Ozawa's representation theorems}\label{subsec:ozawa}

We can now state the two results from the literature on which everything rests.  Throughout, $Z$ is a commutative \AWstar-algebra, $B=\Proj(Z)$, and $Z$ is identified with $Z_{B}$ as above.

The first transfers the \AWstar\ axiom itself across the boundary.

\begin{thm}[{\cite[Theorem~4]{Oza86}}]\label{thm:bounded-section}
Let $N$ be an \AWstar-algebra in $V^{B}$.  Then the bounded global-section algebra of $N$ is an \AWstar-algebra containing $Z_{B}$ in its center as a unital \AWstar-subalgebra.  Conversely, every \AWstar-algebra containing $Z_{B}$ in its center as a unital \AWstar-subalgebra is $Z_{B}$-linearly $*$-isomorphic to the bounded global-section algebra of some \AWstar-algebra in $V^{B}$.
\end{thm}

Two features of Ozawa's proof matter below.  First, in the converse direction the internal \AWstar\ axiom is verified for every internal subset of $N$, and not merely for the ascents $\dot X$ of external subsets.  Given an internal $S$, Ozawa passes to the external set $\{x\in A:\tv{\widehat{x}\in S}=1\}$ and matches the two annihilators against each other using Fact~\ref{fact:quantifiers}.  This is what Lemma~\ref{lem:annihilator} appeals to when it produces an internal projection generating $R_{M}(\dot X)$.  Second, no mixing hypothesis appears.  What the proof needs is that such an algebra is a $Z_{B}$-\cstar-algebra in the sense of \cite{Tak83b}, whose third defining condition is exactly the central mixing of Fact~\ref{fact:dictionary}, and that condition holds automatically by \cite[\S10, Proposition~2]{Ber72}.  We verify it directly in Proposition~\ref{prop:mixing} anyway, since it is short, but it is not an extra assumption.

The second result specializes the first to factors, and is the representation on which both of our theorems rest.

\begin{thm}[{\cite[Theorem~5]{Oza86}}]\label{thm:factor-representation}
Let $N$ be an \AWstar-factor in $V^{B}$.  Then the bounded global-section algebra of $N$ is an \AWstar-algebra with center $Z_{B}$.  Conversely, every \AWstar-algebra with center $Z_{B}$ is $Z_{B}$-linearly $*$-isomorphic to the bounded global-section algebra of some \AWstar-factor in $V^{B}$.
\end{thm}

It is worth pausing on the converse, since it is doing more work than it might look like: the hypothesis is \emph{equality} of the center with $Z_{B}$, not merely containment.  Ozawa obtains it from Theorem~\ref{thm:bounded-section} by observing that the center of a bounded global-section algebra consists of exactly those elements whose ascents are internally central, so that internal factoriality corresponds to the external center being no larger than $Z_{B}$.  This is the point at which taking the full center of $A$ as our scalars, which \S\ref{sec:transfer} shows costs nothing, yields a factor and not merely an algebra.

Since it is the one feature of the cited proofs that our arguments depend on, we verify directly that the internal \AWstar\ axiom holds for \emph{every} internal subset, and not merely for the ascents $\dot X$ of external subsets.

\begin{prop}[Internal \AWstar\ axiom]\label{prop:internal-aw}
    Let $N$ be a name satisfying 
    \[
    \tv{N\text{ is a \cstar-algebra}}=1
    \]
    and let $A$ be its bounded global-section algebra.  If $A$ is an \AWstar-algebra, then $\tv{N\text{ is an \AWstar-algebra}}=1$.
\end{prop}

\begin{proof}
    Let $S$ be a name with $\tv{S\subseteq N}=1$.  We must produce a name $q$ with $\tv{q\text{ is a projection and }R_N(S)=qN}=1$.

    The assertion that, in a \cstar-algebra, $R(S)=R(S^{\flat})$ for
    \[
        S^{\flat}=\bigl\{\max(1,\|x\|)^{-1}x:x\in S\bigr\}\cup\{0\}
    \]
    is a theorem of $\mathrm{ZFC}$: rescaling an element by a nonzero real does not change its annihilator, and $0$ annihilates everything.  By Fact~\ref{fact:bvm}(iii) and (iv) there is a name $S^{\flat}$ with
    \[
        \tv{S^{\flat}\subseteq N,\ \widehat{0}\in S^{\flat},\ \|y\|\leq 1\text{ for every }y\in S^{\flat},\ R_N(S)=R_N(S^{\flat})}=1 .
    \]
    The point of the normalization is that every name $y$ with $\tv{y\in S^{\flat}}=1$ now satisfies $\tv{\|y\|\leq\check{1}}=1$ and is therefore a bounded section; this is what makes Fact~\ref{fact:quantifiers} available below, and it is not available for $S$ itself, whose members may have internally unbounded norm.

    Put $X=\{x\in A:\tv{\widehat{x}\in S^{\flat}}=1\}$. Then $0\in X$, so $X$ is nonempty.  Since $A$ is an \AWstar-algebra there is a projection $e\in A$ with $R_A(X)=eA$.

    Next, we show $\widehat{e}$ annihilates $S^{\flat}$.  As $e=e\cdot 1\in eA=R_A(X)$ we have $xe=0$ for every $x\in X$, so $\tv{\widehat{x}\,\widehat{e}=0}=1$ by Fact~\ref{fact:exactness}(i).  Applying Fact~\ref{fact:quantifiers} with $\varphi(x)\equiv x\in S^{\flat}$ and $\psi(x)\equiv x\widehat{e}=0$, legitimate since $\tv{\widehat{0}\in S^{\flat}}=1$,
    \[
        \tv{\forall x\in N\,(x\in S^{\flat}\Rightarrow x\widehat{e}=0)} =\bigwedge_{x\in X}\tv{\widehat{x}\,\widehat{e}=0}=1 .
    \]
    Hence $\tv{\widehat{e}\in R_N(S^{\flat})}=1$, and since $R_N(S^{\flat})$ is internally a right ideal, $\tv{\widehat{e}N\subseteq R_N(S^{\flat})}=1$.

    For the reverse inclusion, since $\tv{\widehat{0}\in R_N(S^{\flat})}=1$, Fact~\ref{fact:quantifiers} applies with $\varphi(a)\equiv a\in R_N(S^{\flat})$ and $\psi(a)\equiv a\in\widehat{e}N$:
    \[
        \tv{\forall a\in N\,(a\in R_N(S^{\flat})\Rightarrow a\in\widehat{e}N)} =\bigwedge\{\tv{\widehat{a}\in\widehat{e}N}:a\in A,\ \tv{\widehat{a}\in R_N(S^{\flat})}=1\} .
    \]
    Fix such an $a$.  For every $x\in X$ we have $\tv{\widehat{x}\in S^{\flat}}=1$, so $\tv{\widehat{x}\,\widehat{a}=0}=1$ and hence $xa=0$ by Fact~\ref{fact:exactness}(i).  Thus $a\in R_A(X)=eA$, so $ea=a$, so $\tv{\widehat{e}\,\widehat{a}=\widehat{a}}=1$ and $\tv{\widehat{a}\in\widehat{e}N}=1$.  The meet is therefore $1$.

    Combining the two inclusions, $\tv{R_N(S)=R_N(S^{\flat})=\widehat{e}N}=1$, and $\widehat{e}$ is internally a projection by Fact~\ref{fact:exactness}(iii).  As $S$ was arbitrary, the internal \AWstar\ axiom holds with Boolean truth value $1$.
\end{proof}

The remaining ingredient is Sait\^o and Wright's theorem, which we simply cite.

\begin{thm}[{\cite[Corollary~4.7]{SW91}}]\label{thm:factor-normality}
    Every \AWstar-factor is normal.
\end{thm}

\section{The transfer lemmas}\label{sec:transfer}

This section does the real work of the paper: it checks, one piece at a time, that everything needed to cross the Boolean-valued boundary is actually available.  We begin with the hypotheses of the bounded-section representation itself.

\begin{prop}[The full center]\label{prop:full-center}
Let $A$ be an \AWstar-algebra and $Z=Z(A)$.  Then $Z$ is a commutative
\AWstar-algebra, and the inclusion $Z\subseteq A$ is a unital
\AWstar-subalgebra inclusion.  
\end{prop}

Consequently, the full center satisfies the central-subalgebra hypothesis
of Theorem~\ref{thm:bounded-section} without any further assumption of normality or monotone completeness.

\begin{proof}
We check the two claims in turn, starting with the Baer annihilator condition for $Z$.  For the empty subset of $Z$, the right annihilator is $Z=1_Z$.  Let
$S\subseteq Z$ be nonempty.  The \AWstar property of $A$ gives a projection
$e\in A$ such that
\[
 R_A(S)=eA.
\]
Since $S$ is contained in the center, this right ideal is two-sided.  Indeed, if $x\in R_A(S)$, $a\in A$, and $s\in S$,
then centrality of $s$ gives
\[
 s(ax)=a(sx)=0,\qquad s(xa)=(sx)a=0.
\]
Thus $Ae\subseteq eA$.  It follows that
\[
 (1-e)ae=0
\]
for every $a\in A$.  Applying this equality to $a^*$ and taking adjoints
gives $ea(1-e)=0$.  Both off-diagonal corners vanish, so $ae=eae=ea$.
Hence $e\in Z$, and
\[
 R_Z(S)=R_A(S)\cap Z=eZ.
\]
This proves the Baer annihilator condition for $Z$.

For the second claim, let $F\subseteq\Proj(Z)$, and put
\[
 p=\bigvee\nolimits_{\Proj(A)}F.
\]
For every unitary $u\in A$, the order automorphism $\Ad_u$ fixes each
member of $F$.  Uniqueness of the join gives $upu^*=p$.  Thus $p$ commutes
with every unitary in $A$.  To see that this implies centrality without an
external input, note that a self-adjoint contraction $h$ is the real part of the
unitary
\[
 u=h+i(1-h^2)^{1/2}.
\]
After scaling and separating real and imaginary parts, every element of
$A$ is a complex linear combination of unitaries.  Therefore $p\in Z$.

The projection $p$ is consequently an upper bound for $F$ in $\Proj(Z)$.
Every upper bound in $\Proj(Z)$ is also an upper bound in $\Proj(A)$, and
hence dominates $p$.  This proves equality of the two joins.  
\end{proof}

\begin{prop}[Arbitrary bounded central mixing]\label{prop:mixing}
    Let $A$, $Z$, and $B=\Proj(Z)$ be as in Proposition~\ref{prop:full-center}. Let $(b_i)_{i\in I}$ be any pairwise orthogonal family in $B$ with join $1$, and let $(x_i)_{i\in I}$ be a norm-bounded family in $A$.  There is a unique $x\in A$ such that
    \[
        b_ix=b_ix_i\qquad(i\in I).
    \]
\end{prop}

That is: once we commit to the full center as our scalars, there is no obstruction to gluing together a norm-bounded family of elements along any partition of central projections.  As noted after Theorem~\ref{thm:bounded-section}, this is not an extra hypothesis on $A$, but it is convenient to have it available in the form of Fact~\ref{fact:dictionary}, and the proof is short.

\begin{proof}
    By Proposition~\ref{prop:full-center}, $Z$ is a commutative \AWstar-algebra and its inclusion in $A$ preserves arbitrary joins of central projections.  The existence of $x$ is \cite[\S10, Proposition~2]{Ber72}; this is the reference Ozawa uses in \cite[\S3]{Oza86} for condition~(3) in the definition of a $Z$-\cstar-algebra, which is the present statement.

    For uniqueness, suppose that $x,y\in A$ both satisfy these equations and set $d=x-y$, so that $b_id=0$ for every $i$.  Let $f$ be a projection with $R_A(\{d\})=fA$.  Since the $b_i$ are central, $db_i=b_id=0$, so $b_i\in R_A(\{d\})$ and hence $b_i\leq f$ for every $i$.  By Proposition~\ref{prop:full-center} the join of the $b_i$ in $\Proj(A)$ is their join in $\Proj(Z)$, namely $1$, so $f=1$ and $d=d\cdot1=0$.
\end{proof}

The following lemma is the bridge between the abstract center $Z$ and the concrete scalar algebra $Z_{B}$ that the Boolean-valued machinery actually manufactures; without it, we would have two commutative \AWstar-algebras with no stated reason to think of them as the same object.

\begin{lem}\label{lem:commutative}
    Let $Z$ be a commutative \AWstar-algebra and let $B=\Proj(Z)$.  Then $B$ is complete, and $Z$ is $*$-isomorphic to $C(X_B)$, where $X_B$ is the Stone space of $B$, by an isomorphism carrying each projection to the characteristic function of the corresponding clopen set.  The same holds of $Z_{B}$, and consequently $Z\cong Z_{B}$ through a $*$-isomorphism that is the identity on $B$.
\end{lem}

This identification is \cite[Theorem~3.6 and Corollary~3.7]{Oza84} and is presupposed throughout \cite{Oza86}.  We give a proof in the form we need for completeness.

\begin{proof} 
    Since $Z$ is a commutative \AWstar-algebra, it is its own unique maximal abelian self-adjoint subalgebra.  The MASA characterization of \AWstar-algebras \cite[Theorem 1.5]{SW15} therefore implies that $Z$ is monotone complete.  By Gelfand duality, $Z\cong C(X)$ for a compact Hausdorff space $X$, and monotone completeness of $C(X)$ implies that $X$ is extremally disconnected, hence Stonean \cite[Theorems~2.3.7 and 8.2.5]{SWBook15}.

    The projections of $C(X)$ are precisely the characteristic functions of clopen subsets of $X$.  Thus $B=\Proj(Z)$ is the Boolean algebra of clopen subsets of $X$ \cite[\S7, Theorem~1, pp.~40--41]{Ber72}.  Since $X$ is Stonean, this Boolean algebra is complete, and Stone duality identifies $X$ with its Stone space $X_B$.

    The bounded global sections of the internal complex numbers over $B$ form the commutative \AWstar-algebra $C(X_B)$, with $B$ as its projection lattice \cite[\S3, pp.~7 and 10--11]{Oza90}.  Hence $Z_{B}\cong C(X_B)\cong Z$. Under both isomorphisms, each $b\in B$ corresponds to the characteristic function of the same clopen subset of $X_B$, so the composite isomorphism fixes $B$ pointwise.
\end{proof}

Putting the last two propositions and the lemma together, Ozawa's Theorem~\ref{thm:factor-representation} applies to any \AWstar-algebra whatsoever, with its own center as the scalars.

\begin{cor}[Full-center factor representation]\label{cor:factor-rep}
    Let $A$ be an \AWstar-algebra, let $Z=Z(A)$, and let $B=\Proj(Z)$.  In $V^{B}$ there is an internal \AWstar-factor $M$ whose bounded global-section algebra is $Z$-linearly $*$-isomorphic to $A$.  This representation is obtained in $\mathrm{ZFC}$.
\end{cor}

\begin{proof}
Proposition~\ref{prop:full-center} shows that $Z$ is a commutative \AWstar-algebra, that $B$ is complete, and that $Z$ sits in the center of $A$ as a unital \AWstar-subalgebra.  By Lemma~\ref{lem:commutative}, $Z\cong Z_{B}$ through an isomorphism fixing $B$ pointwise, and we identify the two.  Since $Z$ is the full center of $A$, the algebra $A$ is an \AWstar-algebra with center $Z_{B}$, which is exactly the hypothesis of the converse direction of Theorem~\ref{thm:factor-representation}.  That theorem yields an internal \AWstar-factor $M$ whose bounded global sections are $Z$-linearly $*$-isomorphic to $A$.  The Scott--Solovay construction and Theorem~\ref{thm:factor-representation} are $\mathrm{ZFC}$ results, so no additional set-theoretic hypothesis is used.
\end{proof}

\begin{rem}\label{rem:iso-invariance}
From this point on, we identify $A$ with the bounded global-section algebra of $M$ via the isomorphism supplied by Corollary~\ref{cor:factor-rep}, rather than repeating at each use that the two are only isomorphic and not literally equal.  This costs nothing, since every notion considered below (the self-adjoint order, projections, projection-lattice joins, self-adjoint suprema, normality, and monotone completeness) is preserved by any $*$-isomorphism: such a map is linear, sends adjoints to adjoints, and preserves the \cstar-norm and positive cone, so it is in particular an order isomorphism on self-adjoint parts, and consequently it carries every upward-directed family, whether of projections or of arbitrary self-adjoint elements, onto an upward-directed family, together with a least upper bound, when one exists, onto the least upper bound of the image family.  In particular, a result stated for the case where $A$ literally is the bounded global-section algebra of $M$, such as Lemma~\ref{lem:descent-normality} below, applies verbatim to the $A$ that Corollary~\ref{cor:factor-rep} actually produces.
\end{rem}

With the internal factor $M$ in hand, the remaining task is to make sure that whatever happens externally in $A$ (an upward-directed family of projections, an annihilator, a supremum) has a faithful internal counterpart in $M$, and vice versa.  The next three results carry this out.

\begin{lem}[Directed ascent by mixing]\label{lem:directed-ascent}
Let $B$ be a complete Boolean algebra, let $M$ be an internal
\cstar-algebra in $V^{B}$, and let $A$ be its bounded global-section
algebra.  If $D\subseteq\Proj(A)$ is externally upward-directed, then
\[
 \tv{\dot D \text{ is a nonempty upward-directed subset of } \Proj(M)}=1.
\]
\end{lem}

The subtlety here is the one recorded in Lemma~\ref{lem:mixtures}: $\dot D$ contains more than just the displayed elements $\widehat d$.  It also contains every mixed element built by gluing together infinitely many different $d_i$'s along a partition, and those mixed elements need common upper bounds just as much as the displayed ones do.  The proof spends most of its effort on exactly this point.

\begin{proof}
Each displayed $\widehat{d}$, $d\in D$, belongs to $\dot D$ with Boolean truth
value $1$.  Since $D$ is nonempty, so is $\dot D$.  By
Fact~\ref{fact:exactness}(iii), $\tv{\widehat{d}\in\Proj(M)}=1$.  For an
arbitrary name $x$, \eqref{eq:membership} gives
\[
 \tv{x\in\dot D}=\bigvee_{d\in D}\tv{x=\widehat{d}}.
\]
Substitutivity yields
\[
 \tv{x=\widehat{d}}\leq\tv{x\in\Proj(M)}
\]
for every $d$, so $\dot D$ is internally a subset of $\Proj(M)$ with
Boolean truth value $1$.

It remains to include the mixed members in the directedness argument.  Fix
names $x,y$ and $b\in B$ satisfying
\[
 b\leq\tv{x\in\dot D}\meet\tv{y\in\dot D}.
\]
By Lemma~\ref{lem:mixtures} there are a partition $(b_i)_{i\in I}$ of $b$
with $b_i\leq\tv{x=\widehat{d_i}}$ for elements $d_i\in D$, and a partition
$(b'_k)_{k\in K}$ of $b$ with $b'_k\leq\tv{y=\widehat{e_k}}$ for elements
$e_k\in D$.  Let $(c_j)_{j\in J}$ be their common refinement, as in
Lemma~\ref{lem:refinement}, and for each $j$ choose indices $i(j)$ and
$k(j)$ with $c_j\leq b_{i(j)}\meet b'_{k(j)}$.  Upward directedness of $D$
supplies $f_j\in D$ satisfying
\[
 d_{i(j)}\leq f_j,\qquad e_{k(j)}\leq f_j
\]
for every nonzero $c_j$.

Choose $d_0\in D$.  The mixing principle gives a name $z$ such that
\[
 c_j\leq\tv{z=\widehat{f_j}}\ (j\in J),\qquad 1-b\leq\tv{z=\widehat{d_0}}.
\]
The name $z$ belongs to $M$ with Boolean truth value $1$, since it is
locally equal on a partition of $1$ to elements of $M$, and
Lemma~\ref{lem:mixtures} gives $\tv{z\in\dot D}=1$.

By Fact~\ref{fact:exactness}(iii) applied to $d_{i(j)}\leq f_j$ and
$e_{k(j)}\leq f_j$, followed by substitutivity on $c_j$,
\[
 c_j\leq\tv{x\leq z\text{ and }y\leq z}.
\]
Joining over $j$ yields
\[
 b\leq\tv{x\leq z\text{ and }y\leq z}.
\]
Since $x$, $y$, and $b$ were arbitrary, the directedness formula has
Boolean truth value $1$.
\end{proof}

\begin{lem}[Exact right-annihilator and join transfer]\label{lem:annihilator}
    Let $B$ be a complete Boolean algebra, let $M$ be an internal \AWstar-algebra in $V^{B}$, and let $A$ be its bounded global-section algebra.  Let $X\subseteq A$ be arbitrary and let 
    \[
        R_A(X)=eA   
    \]
    for a projection $e\in A$.  Then
    \[
        \tv{R_M(\dot X)=\widehat{e}\, M}=1.
    \]
    Consequently, if $S\subseteq\Proj(A)$ and $p=\bigvee\nolimits_{\Proj(A)}S$, then
    \[
    \tv{\bigvee\nolimits_{\Proj(M)}\dot S=\widehat{p}}=1.
    \]
\end{lem}

This is the lemma that lets us stop worrying about the projection $e$ generating a right annihilator externally and its internal counterpart $q$ possibly drifting apart: they do not.  The two projections turn out to be exactly the same element, viewed on either side of the Boolean-valued boundary.

\begin{proof}
    Since $M$ is internally an \AWstar-algebra, the internal \AWstar axiom applied to the internal subset $\dot X$, together with the maximum principle (Fact~\ref{fact:bvm}(iv)), supplies an internal projection $q$ satisfying
    \[
        \tv{R_M(\dot X)=qM}=1.
    \]
    The assertion that every projection in a \cstar-algebra has norm at most $1$ is a $\mathrm{ZFC}$ theorem about the underlying set, operations, involution, and norm.  It is therefore a set-theoretic statement to which the transfer principle applies, and it gives $\tv{\|q\|\leq1}=1$.  Thus $q$ is a bounded global section, and by Fact~\ref{fact:exactness}(iv) there is $r\in A$ with
    \[
        \tv{\widehat{r}=q}=1,
    \]
    which is a projection by Fact~\ref{fact:exactness}(iii).

    Since $q\in R_M(\dot X)$, every displayed member of $\dot X$ annihilates $q$.  Hence
    \[
        \tv{\widehat{x}\, q=0}=1
    \]
    for every $x\in X$.  Exact descent of multiplication and equality gives $xr=0$, so $r\in R_A(X)=eA$.  Write $r=ea$.  Then $er=r$, and taking adjoints gives $re=r$.  Since $e$ and $r$ are projections, $r\leq e$. Exact ascent of this order relation yields
    \[
         \tv{q\leq\widehat{e}}=1.
    \]

    Conversely, $xe=0$ for every $x\in X$.  For an arbitrary name $y$, substitutivity gives
    \[
        \tv{y=\widehat{x}}\leq\tv{y\,\widehat{e}=0}
    \]
    for each $x\in X$.  Taking the join and using \eqref{eq:membership} gives
    \[
        \tv{y\in\dot X}\leq\tv{y\,\widehat{e}=0}.
    \]
    Thus $\widehat{e}$ annihilates every member of $\dot X$ with Boolean truth value $1$, and therefore $\widehat{e}\in qM$.  Internally, $\widehat{e}=qm$ for some $m$.
    Multiplication by $q$, followed by taking adjoints, gives
    \[
        q\widehat{e}=\widehat{e}=\widehat{e}\, q.
    \]
    Thus $\tv{\widehat{e}\leq q}=1$.  Combining the two inequalities gives $\tv{q=\widehat{e}}=1$, which proves the annihilator identity.

    For the consequence, let $R_A(S)=fA$.  If $s\in S$, then $sf=0$, so $s\leq1-f$.  Hence $p\leq1-f$.  Conversely, $s\leq p$ gives $s(1-p)=0$, so $1-p\in fA$ and $1-p\leq f$.  Therefore $f=1-p$.  The first part of the lemma now gives
    \[
        \tv{R_M(\dot S)=(1-\widehat{p})M}=1.
    \]
    The statement that the join of a set of projections in an \AWstar-algebra is the complement of its right-annihilator generator is a theorem of $\mathrm{ZFC}$: it is the elementary projection-order argument in the preceding paragraph, formalized using only sets, algebraic operations, and the \AWstar axiom. The transfer principle therefore assigns Boolean truth value $1$ to that statement inside $V^{B}$.  Applied to $\dot S$, it identifies its internal projection join as $\widehat{p}$.
\end{proof}

We have now checked everything needed to know that projections and their joins transfer correctly.  The one thing we have not yet touched is normality itself: the statement about \emph{self-adjoint} upper bounds, not just projection upper bounds.  That is the content of the final lemma of this section.

\begin{lem}[Exact descent of normality]\label{lem:descent-normality}
    Let $B$ be a complete Boolean algebra, and let $M$ be an internal normal \AWstar-algebra in $V^{B}$.  If $A$ is the bounded global-section algebra of $M$, then $A$ is a normal \AWstar-algebra.
\end{lem}

\begin{proof}
    We provide two separate proofs: one here and the other in the remark that follows.  The forward direction of Theorem~\ref{thm:bounded-section} makes $A$ an \AWstar-algebra.  Let $D\subseteq\Proj(A)$ be upward-directed, and let
    \[
        p=\bigvee\nolimits_{\Proj(A)}D.
    \]
    Lemma~\ref{lem:directed-ascent} gives Boolean truth value $1$ to the assertion that $\dot D$ is a nonempty upward-directed subset of $\Proj(M)$. Lemma~\ref{lem:annihilator} identifies its internal projection-lattice join:
    \[
        \tv{\bigvee\nolimits_{\Proj(M)}\dot D=\widehat{p}}=1.
    \]
    Let $a\in A_{\mathrm{sa}}$ satisfy $d\leq a$ for every $d\in D$.  By Fact~\ref{fact:exactness}(iii), $\widehat{a}$ is internally self-adjoint and
    \[
        \tv{\widehat{d}\leq\widehat{a}}=1
    \]
    for every $d\in D$.  If $y$ is an arbitrary name, substitutivity and \eqref{eq:membership} give
    \[
        \tv{y\in\dot D}=\bigvee_{d\in D}\tv{y=\widehat{d}}\leq\tv{y\leq\widehat{a}}.
    \]
    Thus $\widehat{a}$ is an internal self-adjoint upper bound for all mixed members of $\dot D$, not only for its displayed ascents, which is exactly the gap that needed closing before internal normality could be applied at all.

    Internal normality of $M$ now gives
    \[
        \tv{\widehat{p}\leq\widehat{a}}=1.
    \]
    We verify the exact descent of this order inequality.  Internally, $\widehat{a}-\widehat{p}$ is positive.  The positive-square-root theorem for \cstar-algebras is a theorem of $\mathrm{ZFC}$, since \cstar-algebras, their norms, and the functional-calculus assertion are set-theoretically encoded.  The transfer principle therefore supplies an internal $s$ satisfying
    \[
        \widehat{a}-\widehat{p}=s^*s
    \]
    with Boolean truth value $1$.  The transferred \cstar norm identity gives $\|s\|^2=\|\widehat{a}-\widehat{p}\|$.  Since $\widehat{a}-\widehat{p}$ is bounded by a ground-model real, $s$ is a bounded global section.  Choose $c\in A$ with $\tv{\widehat{c}=s}=1$.  Exact descent of the displayed identity gives
    \[
        a-p=c^*c\geq0.
    \]
    Therefore $p\leq a$.  Since $a$ was an arbitrary self-adjoint upper bound, $p$ is the least upper bound of $D$ in $A_{\mathrm{sa}}$.
\end{proof}

\begin{rem}
    There is an equivalent supremum-descent approach to Lemma \ref{lem:descent-normality}.  Internal normality supplies the self-adjoint supremum $u$ of $\dot D$, and Lemma~\ref{lem:annihilator} gives $\tv{u=\widehat{p}}=1$.  The reverse direction of the bounded-directed-set result \cite[Lemma~8.2(3)]{Oza90} yields $v$ in the self-adjoint unit ball of $A$ such that $v$ is the external self-adjoint supremum of $D$ and $\tv{\widehat{v}=u}=1$.  Since $p$ is a self-adjoint upper bound, $v\leq p$.  The element $u$ is internally a projection, so $v$ is a projection by Fact~\ref{fact:exactness}(iii).  Minimality of $p$ among projection upper bounds gives $p\leq v$.  Hence $p=v$, which again proves normality.
\end{rem}

\section{Proof of the main theorem}\label{sec:main-proof}

The hard work is behind us.  All that remains is to assemble the representation, the factor theorem, and the transfer lemmas into a single argument, which turns out to take only a few lines.

\begin{proof}[Proof of Theorem~\ref{thm:normality}]
    Let $A$ be an \AWstar-algebra, set $Z=Z(A)$, and put $B=\Proj(Z)$. Corollary~\ref{cor:factor-rep} supplies an internal \AWstar-factor $M$ in $V^{B}$ whose bounded global-section algebra is $Z$-linearly $*$-isomorphic to $A$.

    The assertion that every \AWstar-factor is normal is a theorem of $\mathrm{ZFC}$ by Theorem~\ref{thm:factor-normality}.  After \cstar-algebras, projections, arbitrary projection families, joins, and positivity are encoded as sets, this is a $\mathrm{ZFC}$ statement.  By Fact~\ref{fact:bvm}(iii), the assertion that $M$ is normal therefore has Boolean truth value $1$.

    By Lemma~\ref{lem:descent-normality} applied to $M$, the algebra $A$ is normal.
\end{proof}

This concludes our discussion of normality.  We now turn to monotone completeness, where the same representation does the work again, this time at the level of arbitrary bounded self-adjoint directed families rather than just projections.

\section{Monotone completeness}\label{sec:monotone}

We now prove Theorem~\ref{thm:monotone-reduction}.  All of the work is done by an unconditional biconditional describing exactly how monotone completeness passes across the bounded-section representation, with no assumption about factors anywhere in sight.  This is the natural generalization of the projection-specific transfer of Section~\ref{sec:transfer} to arbitrary bounded self-adjoint directed families, and most of the work below amounts to redoing the previous section's arguments one level of generality up.

\begin{lem}[Self-adjoint directed-family ascent]\label{lem:selfadjoint-directed-ascent}
    Let $B$ be a complete Boolean algebra, let $M$ be an internal unital \cstar-algebra in $V^{B}$, and let $A$ be its bounded global-section algebra.  Suppose $D\subseteq A_{\mathrm{sa}}$ is nonempty and upward directed and $\|d\|\leq K$ for all $d\in D$, where $K$ is a ground-model positive integer.  Then
    \[
        \tv{\dot D\text{ is a nonempty norm-bounded upward-directed subset of } M_{\mathrm{sa}}}=1,
    \]
    and every member of $\dot D$ has norm at most $\check{K}$.
\end{lem}

This is the exact analogue of Lemma~\ref{lem:directed-ascent}, with $\Proj(A)$ replaced by all of $A_{\mathrm{sa}}$.

\begin{proof}
    Choose $d_0\in D$.  The displayed member $\widehat{d_0}$ proves internal nonemptiness.  For any name $x$,
    \[
        \tv{x\in\dot D}=\bigvee_{d\in D}\tv{x=\widehat d}.
    \]
    By Fact~\ref{fact:exactness}, every $\widehat d$ is internally self-adjoint and has norm at most $\check{K}$, so substitutivity gives
    \[
        \tv{x\in\dot D}\leq\tv{x=x^*\text{ and }\|x\|\leq\check{K}}.
    \]

    For directedness, let $x,y$ be names and put $q=\tv{x\in\dot D}\meet\tv{y\in\dot D}$.  Distributivity gives
    \[
        q=\bigvee_{(d,e)\in D\times D}\bigl(\tv{x=\widehat d}\meet\tv{y=\widehat e}\bigr).
    \]
    By Lemma~\ref{lem:refinement}, applied to this family indexed by $D\times D$, these Boolean values may be disjointified into cells $a_{d,e}$ with join $q$.  Upward directedness of $D$ supplies $f_{d,e}\in D$ with $d,e\leq f_{d,e}$.  Mix $\widehat{f_{d,e}}$ on $a_{d,e}$ and $\widehat{d_0}$ on $1-q$.  The resulting name $z$ belongs to $\dot D$ with truth value one.  On $a_{d,e}$, exact ascent and substitutivity yield $x,y\leq z$.  Joining the cells gives
    \[
        q\leq\tv{\text{there is }z\in\dot D\text{ with }x\leq z\text{ and }y\leq z}.
    \]
    This proves internal upward directedness.
\end{proof}

\begin{prop}[Upper-bound ascent and exact order descent]\label{prop:upper-bound-transfer}
    Let $B$ be a complete Boolean algebra, let $M$ be an internal unital \cstar-algebra in $V^{B}$, and let $A$ be its bounded global-section algebra.  For a nonempty external $D\subseteq A_{\mathrm{sa}}$, form $\dot D$ from the canonical ascents with coefficient $1$.  Then, for every $a\in A_{\mathrm{sa}}$,
    \[
        a\text{ upper-bounds }D \iff \tv{\widehat a\text{ upper-bounds }\dot D}=1.
    \]
    If an internal $u$ is the least upper bound of $\dot D$ in the full internal self-adjoint order and $\tv{\|u\|\leq\check{n}}=1$ for a ground-model integer $n$, then $u$ descends to the least upper bound of $D$ in the full external self-adjoint order.
\end{prop}

There is no single result in Section~\ref{sec:transfer} this generalizes directly; rather, it extracts and isolates the ascent-then-descend maneuver that appears embedded inside the proof of Lemma~\ref{lem:descent-normality}, freed from that lemma's extra hypotheses that $M$ be an \AWstar-algebra and that $D$ consist of projections.

\begin{proof}
    Suppose $d\leq a$ for every $d\in D$.  By Fact~\ref{fact:exactness}(iii), $\tv{\widehat d\leq\widehat a}=1$.  For any name $x$, $\tv{x\in\dot D}=\bigvee_{d\in D}\tv{x=\widehat d}$, and substitutivity yields $\tv{x=\widehat d}\leq\tv{x\leq\widehat a}$; taking the join over $d$ shows that $\widehat a$ upper-bounds every member of $\dot D$.

    Conversely, each displayed $\widehat d$ belongs to $\dot D$ with truth value one.  If $\widehat a$ internally upper-bounds $\dot D$, then $\tv{\widehat d\leq\widehat a}=1$ for every $d$, and Fact~\ref{fact:exactness}(iii) gives $d\leq a$.

    Now let $u$ have the asserted least-upper-bound and norm properties. Fact~\ref{fact:exactness}(iv) supplies $v\in A$ with $\tv{\widehat v=u}=1$, and self-adjointness descends.  The upper-bound property of $u$ gives $d\leq v$ for every $d\in D$.  If $a\in A_{\mathrm{sa}}$ is any external upper bound, the first part makes $\widehat a$ an internal upper bound.  Internal leastness yields $\tv{u\leq\widehat a}=1$, so $v\leq a$.  Hence $v=\sup_{A_{\mathrm{sa}}}D$, with no norm restriction on $a$.
\end{proof}

The last ingredient we need handles a genuine complication that did not arise for projections, and has no counterpart in Section~\ref{sec:transfer}: an internally norm-bounded directed family need not be uniformly bounded by any single \emph{ground-model} integer all at once, only by some internal real.  (Every projection is automatically bounded by $1$, which is exactly why this issue never surfaced in Lemma~\ref{lem:directed-ascent}.)  The fix is to slice $B$ into countably many bands, on each of which a uniform ground-model bound does exist.

\begin{thm}[Countable-band form for directed families]\label{thm:directed-band-form}
    Let $B$ be a complete Boolean algebra, let $M$ be an internal unital \cstar-algebra in $V^{B}$, and let $A$ be its bounded global-section algebra. Suppose
    \[
        \tv{\D\text{ is a nonempty norm-bounded upward-directed subset of } M_{\mathrm{sa}}}=1.
    \]
    There are a countable partition $(b_k)_{k\in I}$ of $1$, indexed by a set $I$ of positive integers, and nonempty norm-bounded upward-directed sets $D_k\subseteq(b_kA)_{\mathrm{sa}}$ such that
    \[
        \|d\|\leq k\quad(d\in D_k)
    \]
    and, for every internal name $x$,
    \[
        b_k\meet\tv{x\in\D}=\bigvee_{d\in D_k}\bigl(b_k\meet\tv{x=\widehat d}\bigr).
    \]
    The family $\D$ has an internal least upper bound if and only if every $D_k$ has an external least upper bound $y_k$ in $b_kA$.  The internal supremum is the mixing of the $\widehat{y_k}$'s, and its leastness against arbitrary internal upper bounds requires no global norm bound.
\end{thm}

\begin{proof}
    By Fact~\ref{fact:bvm}(iv), choose an internal positive real $R$ satisfying $\tv{\|x\|\leq R\text{ for every }x\in\D}=1$.  Define $c_k=\tv{R\leq\check{k}}$ and $b_k=c_k\meet(1-c_{k-1})$ (with $c_0=0$), discarding zero cells and letting $I$ be the set of surviving indices.  By \eqref{eq:archimedean} we have $\bigvee_kc_k=1$, so the $b_k$'s form a countable partition, and $b_k\leq c_k$.

    For each $k$, set
    \[
        D_k=\{d\in(b_kA)_{\mathrm{sa}}:b_k\leq\tv{\widehat d\in\D}\}.
    \]
    The maximum principle supplies $x_0$ with $\tv{x_0\in\D}=1$; on $b_k$ its norm is at most $\check{k}$, so mixing $x_0$ on $b_k$ with zero off $b_k$ gives a bounded section, whose descent belongs to $D_k$.  Hence $D_k$ is nonempty.

    If $d,e\in D_k$, internal directedness and the maximum principle give a name $z$ such that $b_k\leq\tv{z\in\D,\ \widehat d\leq z,\ \widehat e\leq z}$. Mixing $z$ on $b_k$ with zero off $b_k$ and descending gives $f\in D_k$ with $d,e\leq f$.  Hence $D_k$ is upward directed.  Membership in $\D$ and $b_k\leq\tv{R\leq\check{k}}$ imply $\|d\|\leq k$.

    We prove the coverage identity.  Its right side is at most its left side by substitutivity.  For the converse, fix $x$ and put $q=b_k\meet\tv{x\in\D}$.  Mix $x$ on $q$, the fixed $x_0$ on $b_k-q$, and zero on $1-b_k$.  The result has norm at most $k$ and descends to some $d\in D_k$.  On $q$, $x=\widehat d$, so $q\leq\bigvee_{d\in D_k}(b_k\meet\tv{x=\widehat d})$.  This proves equality.

    Suppose $\D$ has internal supremum $u$.  On $b_k$, its members are bounded above by $\check{k}1_M$, and a full member is bounded below by $-\check{k}1_M$.  Thus the restriction of $u$ to $b_k$ is bounded. Descend it to $y_k\in(b_kA)_{\mathrm{sa}}$.  The coverage identity makes $y_k$ an upper bound of $D_k$.  If $w\in(b_kA)_{\mathrm{sa}}$ upper-bounds $D_k$, that identity also shows that $\widehat w$ upper-bounds every member of $\D$ on $b_k$.  Internal leastness applies to global upper bounds, so let $t$ be the mixing of $\widehat w$ on $b_k$ with $u$ off $b_k$.  Then $t$ upper-bounds $\D$ with Boolean truth value $1$, so $\tv{u\leq t}=1$ and in particular $b_k\leq\tv{u\leq\widehat w}$.  Since $y_k$ and $w$ both lie in $b_kA$, descending gives $y_k\leq w$.  Hence $y_k=\sup D_k$.

    Conversely, suppose every $D_k$ has supremum $y_k$.  Since $-kb_k\leq y_k\leq kb_k$, the ascents may be mixed to an internal self-adjoint element $u$.  Coverage gives $\tv{x\leq u\text{ for every }x\in\D}=1$.

    Let $t$ be any internal self-adjoint upper bound.  For fixed $k$, disjointify the values $b_k\meet\tv{\|t\|\leq\check{m}}$, whose join is $b_k$ by \eqref{eq:archimedean}, into a partition $(e_{k,m})_m$ of $b_k$. Descend the restriction of $t$ to $e_{k,m}$ as $w_{k,m}$.  For $a\in D_k$, define $W_{k,m}=w_{k,m}+k(b_k-e_{k,m})$.  This upper-bounds $a$ on both complementary bands.  Therefore $y_k\leq W_{k,m}$, so $e_{k,m}y_k\leq w_{k,m}$.  Ascending and joining first in $m$, then in $k$, gives $\tv{u\leq t}=1$.  Thus $u=\sup\D$.
\end{proof}

With this in hand, the biconditional we were after falls out fairly directly, and comes with a small bonus: it does not just say \emph{that} monotone completeness transfers, but pinpoints the largest band on which it holds.

\begin{thm}[Boolean-valued invariance of monotone completeness]\label{thm:monotone-invariance}
    Let $B$ be a complete Boolean algebra, let $M$ be an internal unital \cstar-algebra in $V^{B}$, and let $A$ be its bounded global-section algebra.  Then
    \[
        A\text{ is monotone complete} \iff \tv{M\text{ is monotone complete}}=1.
    \]
    If $c=\tv{M\text{ is monotone complete}}$, then $c$ is the largest $b\in B$ for which $bA$ is monotone complete.
\end{thm}

Only the implication from internal to external monotone completeness is used in the proof of Theorem~\ref{thm:monotone-reduction} below.  We record the converse, which is what Theorem~\ref{thm:directed-band-form} supplies, because it makes the correspondence exact.  The final clause is not needed either, but it identifies the Boolean value that appears in that proof: $c$ is the largest central band on which $A$ is monotone complete.

\begin{proof}
    Assume $\tv{M\text{ is monotone complete}}=1$.  Let $D\subseteq A_{\mathrm{sa}}$ be nonempty, norm bounded, and upward directed, say $\|d\|\leq K$ for all $d\in D$. Lemma~\ref{lem:selfadjoint-directed-ascent} makes $\dot D$ internally nonempty, upward directed, self-adjoint, and $\check{K}$-bounded.  Internal monotone completeness supplies its least upper bound $u$, and since $\dot D$ is $\check{K}$-bounded, $-\check{K}1_M\leq u\leq\check{K}1_M$, so $u$ is a bounded global section. By Proposition~\ref{prop:upper-bound-transfer}, $u$ descends to the least upper bound of $D$ in $A_{\mathrm{sa}}$.  Hence $A$ is monotone complete.

    Conversely, assume $A$ is monotone complete.  We first record that each central corner $bA$ of a monotone-complete algebra $A$ is itself monotone complete: if $E\subseteq(bA)_{\mathrm{sa}}$ is nonempty, norm bounded, and upward directed, and $y=\sup_AE$, choose a bound $K$ for $E$; then $Kb$ is an upper bound for $E$ in $A$, so $y\leq Kb$.  Fix $e_0\in E$; from $e_0\leq y$ and $(1-b)e_0=0$ we get $(1-b)y\geq0$, while $y\leq Kb$ gives $(1-b)y\leq0$; hence $y=by\in bA$, where it remains least.
    
    It is therefore enough to treat names $\D$ for which
    \[
        \tv{\D\text{ is a nonempty norm-bounded upward-directed subset of }M_{\mathrm{sa}}}=1 .
    \]

    For a general name, let $b$ be that truth value and run the argument below in $V^{B\restr b}$, whose bounded global-section algebra is the monotone-complete corner $bA$ by Lemma~\ref{lem:band-sections}. So let $\D$ be an internal nonempty norm-bounded upward-directed subset of $M_{\mathrm{sa}}$.  Theorem~\ref{thm:directed-band-form} gives a countable partition $(b_k)_{k\in I}$ of $1$ and nonempty norm-bounded upward-directed sets $D_k\subseteq(b_kA)_{\mathrm{sa}}$ with the stated coverage identity.

    Applying the corner observation to each $b_kA$, every $D_k$ has a supremum $y_k\in(b_kA)_{\mathrm{sa}}$.  Theorem~\ref{thm:directed-band-form} mixes the $\widehat{y_k}$'s into an internal element $u$ and shows it is the internal least upper bound of $\D$, with leastness verified against an arbitrary internal competitor carrying no ground-model uniform norm bound.  Hence $\tv{M\text{ is monotone complete}}=1$.

    Finally, work below a Boolean element $b\in B$.  By Lemma~\ref{lem:band-sections}, the bounded global sections of $M$ in the relative universe $V^{B\restr b}$ are exactly $bA$, and by Fact~\ref{fact:restriction} the restricted truth value of internal monotone completeness equals the unit $b$ of $B\!\restr\!b$ precisely when $b\leq c$. Applying the biconditional just proved, computed in $V^{B\restr b}$ in place of $V^{B}$, shows that $bA$ is monotone complete exactly when $b\leq c$.  Hence $c$ is the largest such band.
\end{proof}

All that remains is to combine the biconditional above with the restriction principle of Section~\ref{sec:background}.

\begin{proof}[Proof of Theorem~\ref{thm:monotone-reduction}]
    One direction is immediate: an \AWstar-factor is an \AWstar-algebra, so any model of $\mathrm{ZFC}$ containing an \AWstar-factor that is not monotone complete contains an \AWstar-algebra that is not monotone complete.

    For the other direction, suppose $\mathrm{ZFC}+$ ``there is an \AWstar-algebra that is not monotone complete'' is consistent, and work in a model $V$ of that theory.  Let $A$ be an \AWstar-algebra in $V$ that is not monotone complete, let $Z=Z(A)$, and let $B=\Proj(Z)$.  By Corollary~\ref{cor:factor-rep} there is an internal \AWstar-factor $M$ in $V^{B}$ whose bounded global-section algebra we identify with $A$, as in Remark~\ref{rem:iso-invariance}.  Put
    \[
        c=\tv{M\text{ is monotone complete}},\qquad b=1-c.
    \]
    Were $c=1$, Theorem~\ref{thm:monotone-invariance} would make $A$ monotone complete; so $b\not=0$.

    Now pass to $B\!\restr\!b$.  By Fact~\ref{fact:restriction}, truth values in $V^{B\restr b}$ are obtained from those in $V^{B}$ by meeting with $b$, so
    \[
        \tv{M\text{ is an \AWstar-factor}}^{V^{B\restr b}}=b, \qquad \tv{M\text{ is monotone complete}}^{V^{B\restr b}}=b\meet c=0.
    \]
    Thus $V^{B\restr b}$ is a Boolean-valued model of $\mathrm{ZFC}$ in which the sentence ``there is an \AWstar-factor that is not monotone complete'' has Boolean truth value $1$.  Since $b\not=0$, the algebra $B\!\restr\!b$ is nontrivial, so $0\not=1$ there.  By soundness of Boolean-valued semantics, $\mathrm{ZFC}+$ ``there is an \AWstar-factor that is not monotone complete'' is consistent.
\end{proof}

\begin{rem}\label{rem:direct-transfer}
    The contrapositive of Theorem~\ref{thm:monotone-reduction} gives: if $\mathrm{ZFC}$ proves that every \AWstar-factor is monotone complete, then $\mathrm{ZFC}$ proves that every \AWstar-algebra is monotone complete.  This form has a direct proof, which uses only Fact~\ref{fact:bvm}(iii) and no restriction to a band, and we record it because it is the form most likely to be applied.  Suppose $\mathrm{ZFC}$ proves that every \AWstar-factor is monotone complete, and let $A$ be an \AWstar-algebra, with $M$ and $B$ as above.  The assertion that every \AWstar-factor is monotone complete is then a theorem of $\mathrm{ZFC}$, so the transfer principle assigns it Boolean truth value $1$ in $V^{B}$. Specializing to $M$, which is internally an \AWstar-factor, gives $\tv{M\text{ is monotone complete}}=1$, and Theorem~\ref{thm:monotone-invariance} makes $A$ monotone complete.
\end{rem}

Between the two main theorems, normality holds for every \AWstar-algebra without exception, while the monotone completeness conjecture is now equiconsistent with its restriction to factors.

\end{document}